\documentclass[12pt]{amsart}

\usepackage[english]{babel}
\usepackage{pgf}
\usepackage[utf8]{inputenc}

\usepackage{hyperref}

\usepackage{array}
\usepackage{tabularx}
\usepackage{graphicx}
\usepackage{amsmath}
\usepackage{amssymb}
\usepackage{amsthm}

\usepackage{enumerate}

\usepackage{tikz}
\usetikzlibrary{shapes,shapes.geometric,arrows,fit,calc,positioning,arrows,automata,}

\allowdisplaybreaks

\newcommand{\N}{\mathbb{N}}
\newcommand{\Z}{\mathbb{Z}}
\newcommand{\R}{\mathbb{R}}
\newcommand{\Q}{\mathbb{Q}}
\newcommand{\C}{\mathbb{C}}
\renewcommand{\P}{\mathbb{P}}

\newcommand{\abs}[1]{\left| #1 \right|}

\newcommand{\rb}[1]{\left( #1 \right)}

\newtheorem{theorem}{Theorem}[section]
\newtheorem{lemma}[theorem]{Lemma}
\newtheorem{corollary}[theorem]{Corollary}
\newtheorem{proposition}[theorem]{Proposition}
\newtheorem{conjecture}[theorem]{Conjecture}

\theoremstyle{definition}
\newtheorem{definition}[theorem]{Definition}

\theoremstyle{remark}
\newtheorem{remark}{Remark}
\newtheorem*{remark*}{Remark}

\usepackage{enumitem}

\begin{document}
 \selectlanguage{english}
 \title{Multiplicative automatic sequences}
\author[J.\ Konieczny]{Jakub Konieczny}
\address[J.\ Konieczny]{Camille Jordan Institute, 
Claude Bernard University Lyon 1,
43 Boulevard du 11 novembre 1918,
69622 Villeurbanne Cedex, France}
\address{Faculty of Mathematics and Computer Science, Jagiellonian University in Krak\'{o}w, \L{}ojasiewicza 6, 30-348 Krak\'{o}w, Poland\newline}
\email{jakub.konieczny@gmail.com}

\author{Mariusz Lema\'nczyk}
\address[ML]{Faculty of Mathematics and Computer Science, Nicolaus Copernicus University, Chopin street 12/18, 87-100 Toru\'n, Poland}
\email{mlem@mat.umk.pl}

\author{Clemens M\"ullner}
\address[CM]{Institut f\"ur Diskrete Mathematik und Geometrie
TU Wien\\
Wiedner Hauptstr. 8--10\\
1040 Wien, Austria}
\email{clemens.muellner@tuwien.ac.at}

\thanks{ The first-named author was supported by ERC grant ErgComNum 682150 and is currently working within the framework of the LABEX MILYON (ANR-10-LABX-0070) of Universit\'{e} de Lyon, within the program "Investissements d'Avenir" (ANR-11-IDEX-0007) operated by the French National Research Agency (ANR). He also acknowledges support from the Foundation for Polish Science (FNP).
The research of the second-named author was supported by  Narodowe Centrum Nauki grant UMO-2019/33/B/ST1/00364.
The third-named author was supported by European Research Council (ERC) under the European Union’s Horizon 2020 research and innovation programme under the Grant Agreement No 648132. }

 \maketitle
  \begin{abstract}
We obtain a complete classification of complex-valued sequences which are both multiplicative and automatic.
	\end{abstract}

\section{Introduction}
A sequence is automatic if it is accepted by a finite automaton. Such sequences are interesting objects to investigate from various points of view: computer science, information theory, linguistics, as well as dynamics and number theory, see e.g.~\cite{Allouche2003}, \cite{Fogg2002a}, \cite{Queffelec2010}. In particular, a recent motivation to study dynamics of systems determined by automatic sequences appeared in the context of the celebrated Sarnak's conjecture on M\"obius disjointness \cite{Sarnak2011}, and indeed the third author proved its validity for all automatic sequences \cite{Muellner2017}. If the M\"obius function is replaced with an arbitrary bounded multiplicative function then the relevant disjointness question has a more complicated answer. Indeed, the second and the third author in \cite{Lemanczyk2018} proved that if $a:\N\to \C$ is a
primitive\footnote{A morphism (or substitution) $\theta: \mathcal{A} \to \mathcal{A}^{*}$ is called \emph{primitive}, if there exists $n \in \N$, such that for all $a,b \in \mathcal{A}$, $b$ occurs in $\theta^n(a)$. An automatic sequence in turn is called \emph{primitive}, if it is given as the coding of a fixed point of a primitive constant-length substitution.} automatic sequence then it is disjoint with any
bounded, aperiodic, multiplicative\footnote{The sequence $u$ is {\em multiplicative} if $u(mn)=u(m)u(n)$ for all coprime $m,n\in\N$ and $u$ is {\em aperiodic} if $\lim_{N\to\infty}\frac1N\sum_{n=1}^Na(cn+d)=0$ for all $c,d\in\N$.} function $u : \N\to \C$, i.e.\
\[\lim_{N\to\infty}\frac1N\sum_{n=1}^Na(n)\overline{u}(n)=0.\]
The arithmetic M\"obius function is of course aperiodic (and multiplicative), but it is easy to find examples of bounded multiplicative functions which are not aperiodic. In fact, Dirichlet characters provide a rich supply of (completely) multiplicative sequences that are periodic (and hence automatic).

Multiplicative functions which are automatic have been known for a long time.
Probably the most prominent example of a (non-periodic) multiplicative automatic sequence is a variant of the period-doubling sequence\footnote{The variant is given by $a(n) = (-1)^{\nu_2(n)}$, where $\nu_2$ denotes the $2$-adic valuation. The original period-doubling sequence is given by $a'(n) = \nu_2(n+1) \bmod 2$.} which was first studied in~\cite{Damanik2000}.
One can easily guess that multiplicative automatic sequences must possess additional interesting properties.
For example, it is proved in \cite{Schlage-Puchta2011}  that any sequence obeying the slightly stronger assumptions of being automatic, completely multiplicative and not taking the value $0$ must be Besicovitch almost periodic. This has been strengthened in \cite{Lemanczyk2018} to Weyl almost periodicity (in the primitive case).
Relatively recently, multiplicative automatic sequence have been studied in numerous papers, \cite{Allouche2018, Bell2012, Bell2014, Borwein, Borwein2009, Coons2010, Klurman2019a, Klurman2019, Konieczny2019, Li2017, Li2019, Schlage-Puchta2003, Yazdani2001}.
A particular focus was put on the study of the case of completely multiplicative automatic sequences \cite{Allouche2018}, \cite{Li2017} with the complete classification given in \cite{Li2019}.

Returning to the classification problem in full generality (without the assumption of complete multiplicativity), Bell, Bruin and Coons in \cite{Bell2012} (see Conjecture 7.3 therein) conjectured that for each such sequence $a$ there exists an  eventually periodic function $g$ such that $a(p) = g(p)$ for every prime $p$. This conjecture has been proved recently (in a stronger form), independently by Klurman-Kurlberg \cite{Klurman2019a} and the first author \cite{Konieczny2019}. Although, it is also possible to provide an ergodic theory type proof of the aforementioned conjecture following \cite{Lemanczyk2018}, we will not do it here but use the main result of \cite{Konieczny2019} and prove a complete classification of automatic sequences which are multiplicative. The main result of the paper is the following:\footnote{For a prime $p$, $\nu_p(n)$ denotes the $p$-adic valuation, i.e. the largest power of $p$ dividing $n$.}

	\begin{theorem}\label{th_main}
		Let $a: \N \to \C$ be a multiplicative and automatic sequence. Then $a$ is $p$-automatic for some prime $p$ and takes the form
		\begin{align}\label{forma}
			a(n) = f_1(\nu_p(n)) \cdot f_2(n/p^{\nu_p(n)}),
		\end{align}
		where $f_1: \N_0\to \C$ is eventually periodic, $f_1(0) = 1$ and $f_2: \N \to \C$ is multiplicative, eventually periodic and vanishes at all multiples of $p$.
Furthermore, any sequence given by~\eqref{forma} with these conditions is multiplicative and $p$-automatic and this decomposition is unique unless $a(n)$ is eventually periodic.
	\end{theorem}
	\begin{remark}
		The condition that $f_2$ vanishes at the multiples of $p$ seems superfluous, but ensures the uniqueness of the decomposition if $a(n)$ is not eventually periodic.\footnote{Let us assume that we have a decomposition like in Theorem~\ref{th_main}, but without asking $f_2$ to vanish at multiples of $p$. Then we can define a multiplicative function via $g_2(p^k) = 0, g_2(q^k) = f_2(q^k)$ for all primes $q\neq p$ and $k\geq 1$. It follows directly that we can replace $f_2$ by $g_2$ in ~\eqref{forma} and $g_2$ is also eventually periodic as it is just the product of $f_2$ with the indicator function of integers coprime to $p$ which is periodic.} Moreover,  the classification of completely multiplicative automatic sequences obtained by Li~\cite{Li2019} appears as a special case of Theorem~\ref{th_main}, see Section~\ref{s:Remarks} for details.
	\end{remark}
	
	To further motivate Theorem \ref{th_main}, let us put it in a wider context. For a sequence $f \colon \N \to \C$, it is natural to inquire that the associated generating series $F(z) = \sum_{n \geq 1} f(n) z^n$ is algebraic. Since many sequences of number-theoretic interest are multiplicative, the problem of verifying algebraicity of generating series of multiplicative sequences has been thoroughly investigated. The definitive result in this area is the following theorem by Bell, Bruin and Coons~\cite{Bell2012}. It is a generalization of the complex-valued case solved by B\'ezivin~\cite{Bezivin1995}.	

	\begin{theorem}
		Let $K$ be a field of characteristic $0$, let $f: \N \to K$ be a multiplicative function, and suppose that its generating series $F (z) = \sum_{n\geq 1}f (n) z^n$ is algebraic over $K(z)$. Then either there is a natural number $k$ and a periodic multiplicative function $\chi: \N \to K$ such that $f (n) = n^k \chi(n)$ for all $n$, or $f (n)$ is eventually zero.
	\end{theorem}
	
	The analogous question can be posed for sequences taking values in finite fields.
	In this case, there is a classical characterization of algebraic formal power series due to Christol~\cite{Christol1979}.
	\begin{theorem}
		Let $q = p^k$ be a prime power, let $\mathbb{F}_q$ be a finite field of size $q$, and $(a(n))_{n\geq 0}$ a sequence with values in $\mathbb{F}_q$. Then, the sequence $(a(n))_{n\geq 0}$ is $p$-automatic if and only if the formal power series $\sum_{n\geq 0} a(n) X^n$ is algebraic over $\mathbb{F}_q(X)$.
	\end{theorem}
	
	Thus, Theorem~\ref{th_main} can also be applied to obtain a characterization of all multiplicative functions $a: \N \to \mathbb{F}_q$ such that $\sum_{n\geq 1} a(n) X^n$ is algebraic over $\mathbb{F}_q(X)$.

The paper is structured as follows.
In Section 2 we give an overview of relevant basic facts about automatic sequences.
Then we provide some comments on the main theorem in Section 3, to put it into the correct context.
In Section 4 we give some final auxiliary results and review the solution of Bell--Bruin--Coons conjecture.
The proof of the main theorem is then split into two Sections 5 and 6 for the dense and sparse case, respectively.
Section 7 finishes the paper with some remarks.

\subsection*{Notation}
Throughout the paper we let $\N_0$ denote the non-negative integers and $\N$ the positive integers. Furthermore, $\C$ denotes the complex numbers and $p,q$ will always be prime numbers.

\section{Automatic sequences}
In this section we review some basic facts concerning automatic sequences. We assume no familiarity with the subject, and for the convenience of the reader we sketch the proofs of even some well-known facts. For a comprehensive introduction, see \cite{Allouche2003}.

There are several equivalent ways to define automatic sequences: primarily, these are sequences that are accepted by finite automata, or the letter-to-letter codings of fixed points of substitutions of constant length. A classical theorem due to Cobham \cite{Cobham1972} characterizes automatic sequences in terms of their kernels. For a sequence $(f(n))_{n \geq 0}$ and $\lambda \geq 2$, the \emph{$\lambda$-kernel} of $f$ is the set of all the possible restrictions of $f$ to an infinite arithmetic progression whose step is a power of $\lambda$, that is, the set
\begin{equation}\label{eq:def-of-kernel}
  			\left\{(f(n \lambda^k + r))_{n\geq 0} \ : \  k \geq 0, 0 \leq r < \lambda^{k}\right\}.
\end{equation}
  	\begin{lemma} [\cite{Cobham1972}]\label{lem:kernel} Given an integer $\lambda \geq 2$,
  		a sequence $f: \N_0 \to \C$ is $\lambda$-automatic if and only if the $\lambda$-kernel of $f$, given by \eqref{eq:def-of-kernel},   		is finite.
  	\end{lemma}
We can extend the definition of $\lambda$-automaticity to  $\lambda=1$ by declaring that
a sequence $f$ is $1$-automatic if and only if it is eventually periodic. In this case, $f$ is $\lambda$-automatic for any $\lambda \in \N$.  	
  	
  Note that, given $k$ and $r$, the first element of the sequence \eqref{eq:def-of-kernel} is $f(r)$ and since $r$ runs over all natural numbers as $k\to\infty$, any automatic sequence takes only finitely many values.

  The following lemma shows that the class of complex-valued automatic sequences is closed under the addition, multiplication, and any other entry-wise operation.
	\begin{lemma}\label{lem:auto-basic}
		Let $\lambda \in \N$ and let $(f(n))_{n \geq 0}$ and  $(g(n))_{n \geq 0}$ be $\lambda$-automatic sequences.
		\begin{enumerate}[wide]
		\item The sequence $\big( \pi(f(n)) )_{n \geq 0}$ is $\lambda$-automatic for any map $\pi$;
		\item The sequence $\big((f(n),g(n))\big)_{n \geq 0}$ is $\lambda$-automatic.
		\end{enumerate}
	\end{lemma}
	\begin{proof}
	It is enough to show that the $\lambda$-kernels of the relevant sequences are finite. If the $\lambda$-kernel of $f$ and $g$ have cardinalities $N$ and $M$ respectively, then the $\lambda$-kernel of $\pi \circ f$ has at most $N$ elements and the $\lambda$-kernel of $(f,g)$ has at most $NM$ elements.
	\end{proof}
	
It is evident that if $f$ is a $\lambda$-automatic sequence, then so are all the sequences in its $\lambda$-kernel. More generally, one can check that if $f$ is $\lambda$-automatic then so is its restriction to any arithmetic progression $n \mapsto f(an+b)$ ($a \in \N$, $b \in \N_0$). We also have the following classical results about being automatic in two different bases.

\begin{lemma}\label{le_power_lambda}
	For any $\lambda \in \N$ and $k \geq 1$, a sequence is $\lambda$-automatic if and only if it is $\lambda^k$-automatic.
\end{lemma}
\begin{proof}
Indeed, the $\lambda^k$-kernel is contained in the $\lambda$-kernel, while the $\lambda$-kernel is contained in the union of $\lambda^k$-kernels of the sequences $n \mapsto f(\lambda^j n + s)$ with $0 \leq j < k$, $0 \leq s < \lambda^j$.
\end{proof}
 Another classical theorem of Cobham asserts that this equivalence is essentially the only case when a sequence is automatic with respect to two different bases.

	\begin{theorem}[\cite{Cobham1972}]\label{th_cobham}
		If a sequence $f$ is both $\lambda$ and $\mu$-automatic, where $\lambda$ and $\mu$ are multiplicatively independent integers (i.e.\ $\log \lambda/\log \mu \notin \Q$) then $f$ is eventually periodic.
	\end{theorem}


\subsection{The pumping Lemma}	  		
The pumping Lemma (for automatic sequences) is a classical tool, which is most often used to show that some particular sequence is not automatic.
We state here a slightly simplified version of it which can be immediately obtained from the Pumping Lemma for regular languages.
Therefore, we need to define for a word $w = w_0 w_1 \ldots w_n$ over the alphabet $\{0,1,\ldots,\lambda-1\}$ the corresponding integer (so that $w$ is the base $\lambda$ representation), i.e. $[w]_{\lambda} := w_0 \lambda^n + w_1 \lambda^{n-1} + \ldots + w_n \lambda^{0}$.
\begin{lemma}
	Let $f$ be a $\lambda$-automatic sequence. Then there exists $n_0 \in \N$ such that for all $n \geq n_0$ there exist words $u,v,w$ over the alphabet $\{0,1,\ldots, \lambda-1\}$ (where $v$ is non-empty) such that $[uvw]_{\lambda} = n$ and $f([uv^kw]_{\lambda}) = f(n)$ for all $k\geq 0$.\footnote{Here, $uvw$ denotes the concatenation of the words $u$, $v$ and $w$ and $v^k$ denotes the $k$-times concatenation of $v$.}
\end{lemma}

We can rewrite the condition $n = [uvw]_{\lambda}$ as
\begin{align*}
	n = [u]_{\lambda} \cdot \lambda^{\abs{vw}} + [v]_{\lambda} \cdot \lambda^{\abs{w}} + [w]_{\lambda}.
\end{align*}
Similarly, we find
\begin{align*}
	[uv^kw]_{\lambda} &= [u]_{\lambda} \cdot \lambda^{\abs{v^kw}} + [v]_{\lambda} \cdot \lambda^{\abs{v^{k-1}w}} + \ldots + [v]_{\lambda} \cdot \lambda^{\abs{w}} + [w]_{\lambda}\\
		&= [u]_{\lambda} \cdot \lambda^{k \abs{v} + \abs{w}} + [v]_{\lambda} \cdot \lambda^{(k-1)\abs{v}+ \abs{w}} + \ldots + [v]_{\lambda} \cdot \lambda^{\abs{w}} + [w]_{\lambda}\\
		&= [u]_{\lambda} \cdot \lambda^{k \abs{v} + \abs{w}} + [v]_{\lambda} \lambda^{\abs{w}} \frac{\lambda^{k\abs{v}} - 1}{\lambda^{\abs{v}}-1} + [w]_{\lambda}.
\end{align*}

Thus, we have the following version of the Pumping Lemma for automatic sequences, where we put $\ell_1 = \abs{w}, \ell_2 = \abs{v}, \ell_3 = \abs{u}$ and $x = [u]_{\lambda}, y = [v]_{\lambda}, z = [w]_{\lambda}$.
\begin{lemma}\label{le_pumping_automatic}
  		Let $f$ be a $\lambda$-automatic sequence. Then there exists $n_0 > 0$ such that for every integer $n \geq n_0$ there exist integers $\ell_1, \ell_3 \geq 0$, $\ell_2 \geq 1$ and $x < \lambda^{\ell_3}, y < \lambda^{\ell_2}, z<\lambda^{\ell_1}$ such that
  		\begin{align*}
  			n =  x \lambda^{\ell_{1} + \ell_{2}} + y \lambda^{\ell_{1}} + z,
  		\end{align*}
  		and
  		\begin{align*}
  			f(n) = f\rb{x \lambda^{\ell_{1} + k \cdot \ell_{2}} + y \lambda^{\ell_1}  \frac{\lambda^{k \cdot \ell_{2}}-1}{\lambda^{\ell_{2}}-1}+ z}
  		\end{align*}
  		for all $k \geq 0$.
  	\end{lemma}

\section{Comments on the main theorem}
To see Theorem~\ref{th_main} in its proper context, in this section we discuss examples of multiplicative automatic sequences and point out some easy observations on decompositions of the type given by \eqref{forma}.  	
 	
Let $p$ be a prime. Recall that for $n \geq 1$, $\nu_p(n)$ denotes the $p$-adic valuation of $n$, that is, the unique integer $k$ such that $p^k \mid n$ and $p^{k+1} \nmid n$. The map $\nu_p \colon \N \to \N_0$ is completely additive, that is,
  \begin{equation}\label{prz1}
  \nu_p(mn)=\nu_p(m)+\nu_p(n)
  \end{equation}
  for all $m,n \geq 1$. It is a standard observation that if $m,n \geq 1$ and $\nu_p(n) \neq \nu_p(m)$ then
  \begin{equation}\label{prz15}
  \nu_p(m+n) = \min\left( \nu_p(m),\nu_p(n)\right).\end{equation}
  Since $\nu_p$ takes the value $0$ at least at one of any pair of coprime integers, it also follows that
  \begin{equation}\label{prz2}
  \nu_p (mn) = \max(\nu_p(m),\nu_p(n))
  \end{equation}
  for any $m,n \geq 1$ with $(m,n) = 1$.
   As a consequence, we immediately obtain a class of examples of $p$-automatic multiplicative sequences coming from $p$-adic valuations.

  \begin{lemma}\label{l:prz1}
  	Let $f_1 \colon \N_0 \to \C$ be a sequence with $f_1(0) = 1$. Then the sequence $a_1 \colon \N \to \C$ given by $a_1(n) = f_1(\nu_p(n))$ is multiplicative. If $f_1$ is eventually periodic, then $a_1$ is $p$-automatic.
  \end{lemma}
\begin{proof}
The first part of the statement follows directly from \eqref{prz2}. For the second part, we verify that the $p$-kernel of $a_1$ is finite. Consider any sequence $n \mapsto a_1(p^k n + r)$ in the $p$-kernel of $a_1$, where $k \geq 0$ and $0 \leq r < p^k$.

Suppose first that $r \neq 0$, and consequently it can be written in the form $r = p^{\ell} r'$ where $p \nmid r'$ and $0 \leq \ell < k$. Then it follows from \eqref{prz15} that
\[
a_1(p^k n + r) = f_1(\ell)
\]
for all $n \geq 0$. Since $f_1$ takes on finitely many values, there are only finitely many constant sequences in the $p$-kernel of $a_1$ that arise this way.

Suppose next that $r = 0$. Then
\[
a_1(p^k n + r) = f_1(\nu_p(n)+k).
\]
Since $f_1$ is eventually periodic, the number of sequences of the form $m \mapsto f_1(m + k)$ with $k \geq 0$ is finite. As a consequence, the $p$-kernel of $a_1$ is also finite, as needed.
\end{proof}

At the opposite extreme we have examples of $p$-automatic multiplicative sequences which are invariant under multiplication by $p$. Before we proceed, we need the following general observation.

\begin{lemma}\label{le_eventually_periodic}
		Let $f: \N \to \C$ be eventually periodic and multiplicative. Then either
		\begin{enumerate}
			\item $f$ is periodic or
			\item $f$ is finitely supported.
		\end{enumerate}
	\end{lemma}
	\begin{proof}
	Since $f$ is eventually periodic, there exists a periodic sequence $h \colon \N \to \C$ with some period $d \geq 1$ and a threshold $n_0$ such that $f(n) = h(n)$ for all $n \geq n_0$. 

	We next elucidate the connection between $f$ and $h$. For any $n \geq 1$ and any $s \geq n_0/nd$, the integers $n$ and $1+snd$ are coprime, so
\begin{equation}\label{eq:508:1}
	\begin{split}
		h(n) &= h\big( n(1+snd) \big) \\&= f\big( n(1+snd) \big) = f(n)f(1+snd) = f(n)h(1).
	\end{split}
	\end{equation}
If $h(1) = 0$ then $h$ would be identically $0$, in which case $f$ is finitely supported and we are done. Otherwise, $h(1) \neq 0$ and $f(n) = h(n)/h(1)$ is periodic.
	\end{proof}

\begin{lemma}\label{l:prz2}
Let $f_2 \colon \N \to \C$ be eventually periodic and multiplicative. Then the sequence $a_2 \colon \N \to \C$ given by $a_2(n) = f_2(n/p^{\nu_p(n)})$ is multiplicative and $p$-automatic.
\end{lemma}
\begin{proof}
The multiplicativity of $a_2$ follows immediately by the multiplicativity of $f_2$ and the discussion above.
Proceeding as in the proof of Lemma~\ref{l:prz1}, we will show that the kernel of $a_2$ is finite. Pick a sequence $n \mapsto a_2(p^k n + r)$ with $k \geq 0$ and $0 \leq r < p^k$. If $r = 0$ then
\[
	a_2(p^k n + r) = a_2(n)
\]
for all $n \geq 0$, which contributes in total one sequence to the $p$-kernel.

Suppose next that $r \neq 0$ and write $r = p^{\ell} r'$ with $p \nmid r'$. Hence, for all $n \geq 0$, we have
\[
	a_2(p^k n + r) = f_2(p^{k-\ell} n + r').
\]
In this situation it will be convenient to split into two cases, depending on which of the alternatives holds in Lemma~\ref{le_eventually_periodic} applied to $f_2$. If $f_2$ is periodic with a period $d$ then $n \mapsto a_2(p^k n + r)$ coincides with one of the (at most) $d^2$ sequences of the form $n \mapsto f_2(i n + j)$ with $0 \leq i,j < d$. If $f_2$ is finitely supported then for all but finitely many choices of $k$ and $r$, the sequence $n \mapsto a_2(p^k n + r)$ is identically zero on $\N$. In either case, the $p$-kernel of $a_2$ is finite.
\end{proof}


%
%

Combining Lemmas \ref{l:prz1} and \ref{l:prz2} (and the basic observations that $p$-automatic sequences are closed under products), we have just proved the following:

\begin{lemma}\label{l:prz3} Let $f_1 \colon \N_0 \to \C$, $f_2 \colon \N \to \C$ be eventually periodic sequences and assume further that $f_1(0) = 1$ and $f_2$ is multiplicative. Then the sequence
$$
n\mapsto f_1(\nu_p(n))f_2(n/p^{\nu_p(n)})
$$
is multiplicative and $p$-automatic.
\end{lemma}

Further remarks:
\begin{itemize}[wide]
\item The sequence $n\mapsto f_1(\nu_p(n))$ is Toeplitz\footnote{A sequence $\N_0 \ni n \mapsto f(n)$ is Toeplitz if for each $n$ there exists $M \geq 1$ such that $f(n+sM) = f(n)$ for all $s \geq 0$.}, and hence almost periodic\footnote{A sequence is called \emph{almost periodic}, if every finite word that appears in the sequence, appears with bounded gaps.}. Indeed, $\nu_p(n)=\nu_p(n+spn)$ for each $s\geq 0$ as
$\nu_p(n+spn)=\nu_p(n(sp+1))=\nu_p(n)+\nu_p(sp+1)=\nu_p(n)$. Hence, each position $n$ has a period (namely $pn$).
\item If $f_2$ is periodic (with some period $d$) then the sequence $n\mapsto f_2(n/p^{\nu_p(n)})$ is Toeplitz, and hence almost periodic. Indeed, assume that $n=p^\alpha(pm+j)$ with  $0<j<p$, i.e.\ $\nu_p(n)=\alpha$. We have
    $$
    f_2((n+sp^{\alpha+1}d)/p^{\nu_p(n+sp^{\alpha+1}d)})=
    f_2(p^\alpha(pm+psd+j)/p^{\nu_p(p^\alpha(pm+psd+j))})=$$
    $$
    f_2(pm+psd+j)=f_2(pm+j)=f_2(n/p^{\nu_p(n)}).$$
    \item This shows in particular that $n\mapsto f_1(\nu_p(n))f_2(n/p^{\nu_p(n)})$ is a primitive $p$-automatic sequence by~\cite[Theorem 5]{Cobham1972} whenever $f_1$ is eventually periodic and $f_2$ is periodic.
    \item If $f_2(1)=1$ and $f_2(n) = 0$ otherwise\footnote{We consider the case of $f_2$ with finite support.} then
    $f_2(n/2^{\nu_2(n)})$ is 1 if $n=2^k$ and it is~0 otherwise. The automatic sequence we obtain is not, in general, primitive.
    For instance, if $f_1(n)=1$ for all even $n\geq0$ and $f_1(n)= 0$ otherwise, then
    $$
    n\mapsto f_1(\nu_2(n))f_2(n/2^{\nu_2(n)})$$
    takes value~1 at $2^m$ with $m\geq0$ even, and the value 0 otherwise. The corresponding automatic sequence is not almost periodic, and hence not primitive by \cite[Theorem 5]{Cobham1972}.
\end{itemize}

The main result of the paper, i.e.\ Theorem~\ref{th_main}, says that the examples given by Lemma \ref{l:prz3}  above exhaust all possible automatic sequences which are multiplicative.

As a partial converse to Lemma \ref{l:prz3}, any, not necessarily automatic, multiplicative sequence $(a(n))_{n \geq 1}$ admits a decomposition
\begin{equation}\label{eq:form-2}
	a(n) = f_1(\nu_p(n)) f_2(n/p^{\nu_p(n)}),
\end{equation} 	
where $f_1$ and $f_2$ are necessarily given by $f_1(k) = a(p^k)$ for all $k \geq 0$ and $f_2(m) = a(m)$ for all $m \geq 1$ with $p \nmid m$. For concreteness we assume that $f_2(m) = 0$ for all $m \geq 0$ with $p \mid m$, which in particular ensures that $f_2$ is multiplicative.

 Below, we record two facts which imply that if $a$ is additionally $p$-automatic then $f_1$ is eventually periodic and $f_2$ is $p$-automatic.
Hence, the content of Theorem \ref{th_main} is that automatic multiplicative sequences are always automatic with respect to bases that are prime and any $p$-automatic multiplicative sequence $a$ with $a(pn) = a(n)$ for all $n \geq 1$ is eventually periodic.

\begin{lemma}[Corollary 5.5.3 \cite{Allouche2003}]\label{podciag} If $f$ is a $\lambda$-automatic sequence for some $\lambda \in \N$ then the sequence $k \mapsto f(\lambda^k)$ is eventually periodic.
\end{lemma}
\begin{proof}
	Since the $\lambda$-kernel of $f$ is finite, there exist integers $k > \ell \geq 0$ such that $f(\lambda^k n) = f(\lambda^\ell n)$ for all $n \geq 1$. In particular, $f(\lambda^{m}) = f(\lambda^{m+k-l})$ for all $m \geq l$.
\end{proof}

\begin{lemma}\label{lem:auto-power-removal}
	If $f$ is a $p$-automatic sequence, where $p$ is a prime, then the sequence $\tilde{f}$ given by $\tilde{f}(n) = f\left( n/p^{\nu_p(n)} \right)$ is also $p$-automatic.
\end{lemma}	
\begin{proof}
	By Lemma \ref{lem:kernel}, it will suffice to verify that the $p$-kernel of $\tilde{f}$ is finite; in fact, we show that it has at most one element more than the kernel of $f$.
	
	Pick any $k \geq 0$ and $0 \leq r < p^k$. If $r = 0$ then, for all $n \in \N$, we have
	\[ \tilde{f}( p^k n + r) = \tilde f(n).\]
	Moreover, if $r \neq 0$, then we can decompose $r = p^{\ell} r'$, where $p \nmid r'$ which, for any $n \geq 0$, implies that
	\[
		\tilde{f}(p^k n + r) = f(p^{k-\ell} n + r').
	\]
	It remains to notice that the sequence $(f(p^{k-\ell} n + r'))_{n \geq 0}$ belongs to the kernel of $f$.
\end{proof}
\begin{remark}
	The statement above is immediate when automatic sequences are viewed through the lens of automata. Indeed, it is enough to modify a single transition in an automaton which accepts $f$ (reading input from the least significant digit) to obtain an automaton which accepts $\tilde f$.
\end{remark}

As already mentioned, the decomposition given by \eqref{eq:form-2} is essentially unique. We record this and previous observations in the following proposition.

\begin{proposition}\label{lem:unique}
	Let $p$ be a prime and let $a$ be a $p$-automatic multiplicative sequence that is not identically zero. Then there exist unique sequences $f_1$ and $f_2$ such that \eqref{eq:form-2} holds, $f_1$ is eventually periodic, $f_2$ is automatic and multiplicative, $f_1(0) = 1$ and $f_2(n) = 0$ for all $n$ with $p \mid n$.
\end{proposition}
\begin{proof}
The existence of such a decomposition has already been proved. For uniqueness, suppose that $a(n) = f_1(\nu_p(n)) \cdot f_2(n/p^{\nu_p(n)})= g_1(\nu_p(n)) \cdot g_2(n/p^{\nu_p(n)})$ were two distinct decompositions as described above. For any integer $n$ with $p \nmid n$, we have
\[
f_2(n) = f_1(0)f_2(n) = a(n) = g_1(0)g_2(n) = g_2(n).
\]
If $p \mid n$ then $f_2(n) = g_2(n) = 0$. Hence, $f_2 = g_2$. Since $a$ is not identically zero, there exists $n_0$ such that $p\nmid n_0$ and $f_2(n_0) \neq 0$. Taking $n = p^k n_0$ for an arbitrary $k \geq 0$, we conclude that
\[
	f_1(k) f_2(n_0) = a(p^k n_0) = g_1(k)f_2(n_0).
\]
It follows that $f_1 = g_1$.


\end{proof}

	\begin{remark}
		By Lemma~\ref{le_eventually_periodic}, we know that $f_2$ in \eqref{eq:form-2} is either finitely supported, which corresponds to the ``sparse case'' or periodic, which corresponds to the ``dense case'', see a discussion in Section \ref{ssec:BBC}.
	\end{remark}

  \section{Auxiliary lemmata}
  	
  	\subsection{Dirichlet characters}\label{subsec_dirichlet}
  	In this subsection we recall some basic properties of Dirichlet characters.
  	\begin{definition}
  		We call a function $\chi : \Z \to \C$ a \emph{Dirichlet character of modulus $k$} if:
  		\begin{itemize}
  			\item $\chi(n) =  \chi(n+k)$ for all $n \in \Z$,
  			\item $\chi(n) = 0$ if and only if $(n,k) > 1$,
  			\item $\chi(m n) = \chi(m) \chi(n)$ for all $m, n \in \Z$.
  		\end{itemize}
  	\end{definition}
  Of course, a Dirichlet character of modulus $k$ is determined by a character of the multiplicative group $(\Z/k\Z)^\ast$ (which takes values in roots of unity of degree $\phi(k)$, where $\phi$ denotes the Euler totient function).
  The Dirichlet character determined by the trivial (constant equal to~1) character of $(\Z/k\Z)^\ast$ is called the {\em principal character} and is denoted by $\widetilde{\chi}$. Note that if $d|k$ and $\eta$ is a Dirichlet character of modulus $d$ then
$$
\chi:=\eta\cdot \widetilde{\chi}$$
is a Dirichlet character of modulus $k$. In this case, we say that $\chi$ is a character {\em induced} from $\eta$.
  	
  	Let $\chi_1, \chi_2$ be Dirichlet characters of modulus $k_1, k_2$ respectively, where $(k_1, k_2) = 1$.
  	Then one sees directly that $\chi_1 \cdot \chi_2$ is a Dirichlet character of modulus $k_1 \cdot k_2$ as the product of (completely) multiplicative functions is again (completely) multiplicative.
  	
  	Moreover, if we additionally assume that $(k_1,k_2)=1$ then this reasoning can be reversed. One can also decompose a Dirichlet character $\chi$ of modulus $k = k_1 k_2$ as the product of Dirichlet characters of modulus $k_1$ and $k_2$. Indeed, we define
  	\begin{align*}
  		\chi_{k_i}(n) := \chi(n_{i}),
  	\end{align*}
  	where $n_{i}$ is uniquely defined modulo $k_1 k_2$ by
  	\begin{align*}
  		n_{i} &\equiv n \bmod k_i\\
  		n_{i} &\equiv 1 \bmod k/k_i.
  	\end{align*}
  	One can easily see\footnote{By checking that $(n+k_i)_{i} \equiv n_{i} \bmod{k_1k_2}$ and $(mn)_{i} \equiv m_{i}n_{i} \bmod{k_1 k_2}$.} that $\chi_{k_i}$ is a Dirichlet character of modulus $k_i$ and $\chi(n) = \chi_{k_1}(n) \cdot \chi_{k_2}(n)$ for all $n\in\Z$.

  	\subsection{Solution of Bell--Bruin--Coons conjecture}\label{ssec:BBC}
  	 \begin{conjecture}[Bell-Bruin-Coons]
		For any multiplicative automatic function $f:\N\to \C$ there exists an eventually periodic function $g:\N\to\C$ such that $f(p) =g(p)$ for all primes $p$.
	\end{conjecture}
This conjecture was recently solved independently by the first author~\cite{Konieczny2019} and Klurman and Kurlberg~\cite{Klurman2019}. We recall the slightly stronger form obtained in \cite{Konieczny2019}.
  	\begin{theorem}[Theorem 1 of~\cite{Konieczny2019}]\label{th_konieczny}
		If $a:\N \to \C$ is an automatic multiplicative sequence then there exist a threshold $p_*$ and a function $\chi: \N \to \C$ which is either a Dirichlet character or identically zero such that $a(n) = \chi(n)$ for all $n \in \N$ not divisible by any prime $p < p_*$.
	\end{theorem}
It will be convenient to refer to a multiplicative automatic sequence as being \emph{sparse} or \emph{dense} if the function $\chi$ in the theorem above is identically zero or a Dirichlet character, respectively. Likewise, we will refer to the corresponding cases of our main theorem as the sparse case and the dense case.
	
We note that if $k$ is a modulus of $\chi$ in Theorem~\ref{th_konieczny}, then replacing $\chi$ by the induced character of modulus $k' ={\rm lcm}(k, p_{\ast}!)$ (and replacing $k$ by $k'$), we conclude that $a(n)=\chi(n)$ for all $n\in\N$ coprime to $k$. This lets us reformulate the theorem above in the following form, which will be more convenient.
	
	\begin{corollary}\label{c:KonMul}
		Let $a: \N \to \C$ be an automatic multiplicative sequence.
		Then there exist coprime integers $h$ and $\lambda$ and a sequence $\chi \colon \N \to \C$ which is either identically zero or a Dirichlet character of modulus $h \lambda$ such that $a$ is $\lambda$-automatic  and $a(n) = \chi(n)$ for all $n \in \N$ that are coprime to $h \lambda$.
	\end{corollary}
	\begin{proof}
	Pick a base $\lambda_0 \in \N$ such that $a$ is $\lambda_0$-automatic. By the remark above, there exist an integer $k \in \N$ and a sequence $\chi_0$ which is either identically zero or a Dirichlet character of modulus~$k$ such that $a(n) = \chi_0(n)$ for all $n$ coprime to $k$.
Now, we can decompose $k = k_1 \cdot k_2$ so that $(\lambda_0, k_2) = 1$ and $k_1 \mid \lambda_0^{\ell}$ for some $\ell \in \N$.
		We take $\lambda := \lambda_0^{\ell}$, $h := k_2$ and see that $a$ is $\lambda$-automatic by virtue of Lemma~\ref{le_power_lambda}. Furthermore, inducing to modulus $h\lambda$, we obtain $\chi$ which is either identically zero or a Dirichlet character of modulus $h \lambda$ and the result follows.
	\end{proof}
	
	Since the integer $h\lambda$ will play a crucial role, it is convenient to make the following definition.
	
	\begin{definition}\label{d:def} We let $\mathcal{C}$ denote the set of integers coprime to $h \lambda$.
	\end{definition}

We also recall that the terms ``sparse'' and ``dense'' were introduced above.	
We would like to stress that in the dense case $a$ must be primitive (in fact, $a$ must be Toeplitz, see the remarks after Lemma~\ref{l:prz3}), while in the sparse case it is not.
  	
\section{Proof of the Main Theorem in the dense case}

Let $a$ be an automatic multiplicative sequence, as in the assumptions of Theorem \ref{th_main}, and let $h,\lambda$ and $\chi$ be given as in Corollary \ref{c:KonMul}. In this section we prove Theorem \ref{th_main} in the dense case, meaning that we assume that $\chi$ is a Dirichlet character of modulus $h\lambda$. It will be convenient to denote $\alpha(p) := \nu_p(h\lambda)$ for $p$ prime.

As we have seen in Subsection~\ref{subsec_dirichlet}, we can decompose $\chi$ into the product of Dirichlet characters of coprime moduli:
\[
\chi = \chi_h \cdot \chi_{\lambda} = \prod_{p \mid h \lambda} \chi_{p^{\alpha(p)}}.
\]
As a consequence, for $n \in \mathcal{C}$, we have
    		\begin{equation}\label{dodana1}
    			a(n) = \chi(n) = \prod_{p \mid h\lambda} \chi_{p^{\alpha(p)}}(n).
    		\end{equation}

For a general $n \in \N$, we can find a decomposition
\begin{equation}\label{eq:n=product}
    		n = n' \cdot \prod_{p \mid h\lambda} p^{\nu_p(n)}
\end{equation}
 as the product of prime divisors of $h \lambda$ and an element of $\mathcal{C}$, with $n'$ given by
\begin{equation}\label{eq:def-of-n'}
	n' = n/ \prod_{p \mid h\lambda} p^{\nu_p(n)}.
\end{equation}
Since $a$ is multiplicative and the factors $n'$ and $p^{\nu_p(n)}$ (for $p \mid h\lambda$) in \eqref{eq:n=product} are pairwise coprime, it follows that
    \begin{align}\label{eq_decompose_a}
      a(n) = a(n') \cdot \prod_{p \mid h\lambda } a(p^{\nu_{p}(n)}).
    \end{align}
We note that $n'$ given by \eqref{eq:def-of-n'} belongs to $\mathcal{C}$. Hence, we can replace the first occurrence of $a$ in the decomposition \eqref{eq_decompose_a} with $\chi$.
    	This and \eqref{dodana1} allows us to rewrite \eqref{eq_decompose_a} as follows
    	\begin{align}\label{eq_decompose_a_2}
    	a(n) &=
    	 \prod_{p \mid h\lambda } \chi_{p^{\alpha(p)}}
    	 \left( \frac{n}{ \prod_{q \mid h\lambda} q^{\nu_q(n)} } \right)
    	 \cdot \prod_{p \mid h\lambda } a(p^{\nu_p(n)}).
    	\end{align}

Our next goal is to simplify the decomposition given in \eqref{eq_decompose_a_2} above. Towards this end, we introduce a new piece of notation (cf.\ notation $n_{i}$ in Section~\ref{subsec_dirichlet}
). For $p \mid h\lambda$, we let $\bar{p}$ be an integer determined uniquely modulo $h\lambda$ by the following system of congruences:
	\begin{align}\label{eq_def_overline}
	\begin{split}
		\overline{p} &\equiv 1 \bmod p^{\alpha(p)}\\
		\overline{p} &\equiv p \bmod h \lambda/p^{\alpha(p)}.
	\end{split}
	\end{align}
This definition is set up so that, in particular,  for each $p \mid h\lambda$, we have
\begin{equation}\label{kreska}
 \chi_{h\lambda/p^{\alpha(p)}}(p)=\chi(\overline{p}).
\end{equation}

We are now ready to write \eqref{eq_decompose_a_2} in a more condensed form.
    	\begin{proposition}\label{pr_decomposition_1}
    		With the notation from \eqref{eq_def_overline}, for all $n \in \N$, we have
    		\begin{align*}
    			a(n) = \prod_{p \mid h\lambda} \chi_{p^{\alpha(p)}}\rb{\frac{n}{p^{\nu_{p}(n)}}}
    			\cdot  \frac{a(p^{\nu_{p}(n)})}{\chi(\overline{p})^{\nu_{p}(n)}}.
    		\end{align*}
    	\end{proposition}
\begin{proof}
It follows from \eqref{eq_decompose_a_2} and the fact that Dirichlet characters are completely multiplicative that
\begin{align}\label{eq:51:1}
   a(n) &= \prod_{p \mid h\lambda } \rb{\chi_{p^{\alpha(p)}}\rb{\frac{n}{p^{\nu_p(n)}}} a(p^{\nu_p(n)}) \prod_{\substack{ q \mid h\lambda \\ q \neq p}} \chi_{p^{\alpha(p)}}(q)^{-\nu_q(n)}}.
\end{align}
Exchanging the order of multiplication in the innermost product yields:
\begin{align}\label{eq:51:2}
 \prod_{p \mid h\lambda } \rb{ \prod_{\substack{ q \mid h\lambda \\ q \neq p}} \chi_{p^{\alpha(p)}}(q)^{\nu_q(n)}} & =
  \prod_{q \mid h\lambda } \rb{ \prod_{\substack{ p \mid h\lambda \\ p \neq q}} \chi_{p^{\alpha(p)}}(q)^{\nu_q(n)}} \\
 &= \prod_{q \mid h \lambda} \chi_{h\lambda/q^{\alpha(q)}}\rb{q}^{\nu_q(n)}.
\end{align}
It remains to recall \eqref{kreska} and insert \eqref{eq:51:2} into \eqref{eq:51:1}.
\end{proof}

The factors in the decomposition produced by Proposition \ref{pr_decomposition_1} correspond to different prime divisors $p$ of $h\lambda$. We are already satisfied with the factors coming from $p \mid \lambda$ and proceed to prove additional properties for $p \mid h$.
Therefore, we will use the following lemma. 
\begin{lemma}\label{lem_const_to_periodic}
	Let $p$ be a prime number and $f: \N \to \C$ be such that $\gamma \mapsto f(p^{\gamma})$ is eventually constant. Then
	\begin{align*}
		n \mapsto f(p^{\nu_p(n)})
	\end{align*}
	is periodic.
	
	If $\gamma \mapsto f(p^{\gamma})$ is eventually equal to zero, then 
	\begin{align}\label{eq_lem_const_to_periodic}
		n \mapsto f(p^{\nu_p(n)}) \cdot \chi\rb{\frac{n}{p^{\nu_p(n)}}}
	\end{align}
	is periodic for any Dirichlet character $\chi$ of modulus $p^{\alpha}$, for an arbitrary $\alpha \in \N$.
\end{lemma}
\begin{proof}
	We denote by $\gamma_p$ the least integer such that the sequence $\gamma\mapsto f(p^\gamma)$ is constant for $\gamma\geq \gamma_p$. If $\nu_p(n)\geq \gamma_p$, then also $\nu_p(n+kp^{\gamma_p})\geq\gamma_p$ for each $k\geq 0$, so the sequence
\begin{align}\label{eq_periodic_1}
	k\mapsto f(p^{\nu_p(n+kp^{\gamma_p})})
\end{align}
 is also constant.
If $\nu_p(n)<\gamma_p$ then
$$
\nu_p(n+kp^{\gamma_p})=\nu_p(n)
$$
for each $k\geq 0$, so the more the sequence~\eqref{eq_periodic_1} is constant. It follows that the sequence
$$
n\mapsto f(p^{\nu_{p}(n)})$$
is $p^{\gamma_p}$ periodic.

Let us now assume that $f(p^{\gamma}) = 0$ for all $\gamma \geq \gamma_p$.
Suppose first that $\nu_p(n) < \gamma_p$. 
We see directly that $\nu_p(n + kp^{\gamma_p+\alpha}) = \nu_p(n)$ for any $k \in \N$.
Thus, 
\begin{align*}
	f(p^{\nu_p(n+kp^{\gamma_p + \alpha})}) \cdot \chi\rb{\frac{n+kp^{\gamma_p+\alpha}}{p^{\nu_p(n+kp^{\gamma_p+\alpha})}}}
		&= f(p^{\nu_p(n)}) \cdot \chi\rb{\frac{n}{p^{\nu_p(n)}} + kp^{\gamma_p + \alpha-\nu_p(n)}}\\
		&= f(p^{\nu_p(n)}) \cdot \chi\rb{\frac{n}{p^{\nu_p(n)}}},
\end{align*}
where the last equality holds, since $\chi$ is $p^{\alpha}$-periodic.
Suppose now that $\nu_p(n) \geq \gamma_p$.
Obviously, $\nu_p(n+kp^{\gamma_p+\alpha}) \geq \gamma_p$, for every $k \in \N$.
Thus,
\begin{align*}
	f(p^{\nu_p(n+kp^{\gamma_p + \alpha})}) \cdot \chi\rb{\frac{n+kp^{\gamma_p+\alpha}}{p^{\nu_p(n+kp^{\gamma_p+\alpha})}}} = 0
\end{align*}
and we have shown in total, that \eqref{eq_lem_const_to_periodic} is periodic with period $p^{\gamma_p + \alpha}$.
\end{proof}

\begin{lemma}\label{lem:a(q^gamma)-ev-constant}
	For any prime $q \mid h$, the sequence
	\begin{align*}
		n \mapsto \chi_{q^{\alpha(q)}}\rb{\frac{n}{q^{\nu_{q}(n)}}}
    			\cdot  \frac{a(q^{\nu_{q}(n)})}{\chi(\overline{q})^{\nu_{q}(n)}}
	\end{align*}
	is periodic.
\end{lemma}

    \begin{proof}
	The proof splits into two parts.
	We first show that $\gamma \mapsto a(q^{\gamma})/\chi(\overline{q})^{\gamma}$ is eventually constant.
	Then we distinguish the case where this sequence eventually equals zero.
	If this is the case, then we are done by the second part of Lemma~\ref{lem_const_to_periodic}.
	If the sequence is not eventually equal to zero, then we also need to show that $\chi_{q^{\alpha(q)}}$ is trivial, i.e. constant and equal to $1$ on integers coprime to $q$.
	In this case we apply the first part of Lemma~\ref{lem_const_to_periodic} and see that
	\begin{align*}
		n \mapsto \frac{a(q^{\nu_{q}(n)})}{\chi(\overline{q})^{\nu_{q}(n)}}
	\end{align*}
	is periodic.
	
	To prove the claims, we consider the finite $\lambda$-kernel of $a$. It follows by the pigeonhole principle that there exist $k\geq 1$ and $r_1 < r_2 < \lambda^k$ such that
      \begin{align}\label{eq_r_1_r_2}
        r_1 \equiv r_2 \equiv 1 \bmod h \lambda
      \end{align}
      and
      \begin{align}\label{eq_a_r_i_equal}
        a(n \lambda^{k} + r_1) = a(n \lambda^{k} + r_2)
      \end{align}
      for all $n \in \N$.
      By the LHS congruence in~\eqref{eq_r_1_r_2} it follows  that there exists $\beta\geq\alpha(q)$ such that
	\begin{align}\label{eq_diff_r_i}
r_2 - r_1 = q^{\beta} r,
	\end{align}
 where $(r,q) = 1$.

	We will show that the sequence $(a(q^\gamma)/\chi(\overline{q})^{\gamma})$ is constant for $\gamma\geq \beta +\alpha(q)$.
Let us fix $\gamma\geq \beta+\alpha(q)$ and $0<s<q^{\alpha(q)}$, where $(s,q) = 1$. Using the fact that $(\lambda,h)=1$ and then the Chinese Remainder Theorem, we find some $n$ such that
      \begin{align}\label{dodana2}
	\begin{split}
        		n \lambda^k + r_1 &\equiv s\cdot q^{\gamma} \bmod q^{\gamma + \alpha(q)}\\
		n \lambda^k + r_1 &\equiv 1 \bmod h/q^{\alpha(q)}.
	\end{split}
      \end{align}
	We will show that
	\begin{align}\label{eq_eval_n_r_1}
		a(n\lambda^k+r_1)=\chi_{q^{\alpha(q)}}(s) \cdot a(q^\gamma)/\chi(\overline{q})^\gamma
	\end{align}
	and
	\begin{align}\label{eq_eval_n_r_2}		a(n\lambda^k+r_2)=a(q^\beta)/\chi(\overline{q})^{-\beta}
\chi_{q^{\alpha(q)}}(r).
	\end{align}
	Suppose that \eqref{eq_eval_n_r_1} and \eqref{eq_eval_n_r_2} have already been proved. Then, by \eqref{eq_eval_n_r_2}, $a(n\lambda^k+r_2)$ is independent of $n$, therefore, independent of $s$ and $\gamma$.
	If $a(q^{\beta}) = 0$, then we conclude by \eqref{eq_a_r_i_equal} and \eqref{eq_eval_n_r_1} that $a(q^{\gamma})/\chi(\overline{q})^{\gamma}$ is eventually equal to zero.
	If $a(q^{\beta}) \neq 0$, then we see that $a(n\lambda^k + r_1) \neq 0$ and it is independent of $s$ and $\gamma$, showing that $\chi_{q^{\alpha(q)}}$ is trivial and $a(q^{\gamma})/\chi(\overline{q})^{\gamma}$ is eventually constant.

	We will now show~\eqref{eq_eval_n_r_1}  and~\eqref{eq_eval_n_r_2} (we proceed similarly to \eqref{eq_decompose_a_2}). Denote $\nu_q(n\lambda^k+r_i)=:m_i$, $i=1,2$. Our aim will be to determine $m_i$ and to show that
$((n\lambda^k+r_i)/q^{m_i},h\lambda)=1$ which, by the multiplicativity  of $a$, yields
	\begin{align*}
		a(n\lambda^k + r_i) = a(q^{m_i}) \cdot a\rb{\frac{n \lambda^k + r_i}{q^{m_i}}} = a(q^{m_i}) \cdot \chi\rb{\frac{n \lambda^k + r_i}{q^{m_i}}}.
	\end{align*}
	Then we decompose $\chi = \chi_{\lambda} \cdot \chi_{q^{\alpha(q)}} \cdot \chi_{h/q^{\alpha(q)}}$ and it only remains to determine the residues of $\rb{n\lambda^k + r_i}/{q^{m_i}}$ modulo $\lambda, q^{\alpha(q)}$ and $h/q^{\alpha(q)}$.

	We now determine $m_1$ by showing that $m_1=\gamma$ and also show that $((n\lambda^k+r_1)/q^{\gamma},h\lambda)=1$.
	Indeed, by \eqref{eq_r_1_r_2} and \eqref{dodana2}, we have
	\begin{align*}
		\frac{n\lambda^k + r_1}{q^{\gamma}} &\equiv s \bmod q^{\alpha(q)},
	\end{align*}
	\begin{align*}
		\begin{split}
		\frac{n\lambda^k + r_1}{q^{\gamma}} \equiv q^{-\gamma} \bmod \lambda,
		\end{split}
		\begin{split}
		\frac{n\lambda^k + r_1}{q^{\gamma}}\equiv q^{-\gamma} \bmod h/q^{\alpha(q)},
		\end{split}
	\end{align*}
	where $q^{-\gamma}$ on the right hand side congruences means, respectively, the reciprocals of $q^\gamma$ in $(\Z/\lambda\Z)^\ast$ and $(Z/(h/q^{\alpha})\Z)^\ast$.
	Thus, we have by the discussion above, the multiplicativity of Dirichlet characters and the definition of $\overline{q}$,
	\begin{align*}
		a(n\lambda^k + r_1) = a(q^{\gamma})\cdot \chi_{\lambda}(q)^{-\gamma} \chi_{q^{\alpha(q)}}(s) \cdot \chi_{h/q^{\alpha(q)}}(q)^{-\gamma} = \chi_{q^{\alpha(q)}}(s) \frac{a(q^{\gamma})}{\chi(\overline{q})^{\gamma}},
	\end{align*}
	which shows~\eqref{eq_eval_n_r_1}.

	As $\gamma \geq \beta + \alpha(q)$, we have $\nu_{q}(n \lambda^k + r_2) = \beta$. Indeed, we find similarly to the previous computation (using \eqref{eq_r_1_r_2}, \eqref{eq_diff_r_i} and \eqref{dodana2}) that
	\begin{align*}
		\frac{n\lambda^k + r_2}{q^{\beta}} = \frac{n\lambda^k + r_1 + r q^{\beta}}{q^{\beta}} &\equiv r \bmod q^{\alpha(q)}
	\end{align*}
	\begin{align*}
		\begin{split}
		\frac{n\lambda^k + r_2}{q^{\beta}} &\equiv q^{-\beta} \bmod \lambda
		\end{split}
		\begin{split}
		\frac{n\lambda^k + r_2}{q^{\beta}} &\equiv q^{-\beta} \bmod h/q^{\alpha(q)},
		\end{split}
	\end{align*}
	which shows~\eqref{eq_eval_n_r_2} and the proof is complete.
\footnote{We note that there is a slightly simpler proof by using more precise results of the first author, i.e.~\cite[Corollary 3.5]{Konieczny2019}.}
   \end{proof}

\subsection{Prime base}
	Now, we are in a position to prove Theorem~\ref{th_main} in the case, where $\lambda$ is a prime power.
	First, we recall that if $\lambda = p^{\alpha}$ for some prime $p$ and $\alpha \geq 1$, then $a$ being $\lambda$-automatic is equivalent to $a$ being $p$-automatic thanks to Lemma~\ref{le_power_lambda}.
Moreover, we recall that we continue to assume that we are in the dense case (see Subsection \ref{ssec:BBC}).
	\begin{proposition}\label{pr_v_1} If $\lambda = p^\alpha$ is a power of a prime $p$ then the automatic sequence $a$ can be written in the form
    		\begin{align}\label{eq:546:0}
    			a(n) = f_1(\nu_p(n)) \cdot f_2\rb{\frac{n}{p^{\nu_p(n)}}},
    		\end{align}
    		where $f_1(0) = 1$,  $f_1$ is eventually periodic and $f_2$ is multiplicative and periodic.
    	\end{proposition}
    	\begin{proof}
		We recall that by Proposition \ref{pr_decomposition_1},
		\begin{align*}
			a(n) = \chi_{p^{\alpha(p)}}\rb{\frac{n}{p^{\nu_{p}(n)}}}
    			\cdot  \frac{a(p^{\nu_{p}(n)})}{\chi(\overline{p})^{\nu_{p}(n)}} \cdot \prod_{q \mid h} \chi_{q^{\alpha(q)}}\rb{\frac{n}{q^{\nu_{q}(n)}}}
    			\cdot  \frac{a(q^{\nu_{q}(n)})}{\chi(\overline{q})^{\nu_{q}(n)}}.
		\end{align*}
		By Lemma~\ref{lem:a(q^gamma)-ev-constant} and Lemma~\ref{l:prz3} we can replace the last product by a periodic and multiplicative function $g$, which gives in total
		\begin{align*}
			a(n) &= \chi_{p^{\alpha(p)}}\rb{\frac{n}{p^{\nu_{p}(n)}}}
    			\cdot  \frac{a(p^{\nu_{p}(n)})}{\chi(\overline{p})^{\nu_{p}(n)}} \cdot g(n)\\
				&= \chi_{p^{\alpha(p)}}\rb{\frac{n}{p^{\nu_{p}(n)}}} g\rb{\frac{n}{p^{\nu_{p}(n)}}}
    			\cdot  \frac{a(p^{\nu_{p}(n)})}{\chi(\overline{p})^{\nu_{p}(n)}} \cdot g(p^{\nu_p(n)}).
		\end{align*}
		Thus, we choose $f_1 (\gamma) = a(p^{\gamma})/\chi(\overline{p})^{\gamma} \cdot g(p^{\gamma})$ and $f_2 = \chi_{p^{\alpha(p)}} \cdot g$. We see directly, that $f_2$ fulfills the requirements of the proposition and it only remains to consider $f_1$. As $a$ is automatic we conclude that $a(1) = 1$ and, therefore also $f_1(0) = 1$. 
The function $f_1$ is eventually periodic, as it is the product of three periodic functions:
we see by Lemma~\ref{podciag} that $\gamma \mapsto a(p^{\gamma})$ is eventually periodic. The function $\gamma \mapsto 1/\chi(\overline{p})^{\gamma}$ is periodic, as $\chi(\overline{p})$ is a root of unity.
We recall that $g$ is periodic, say with period $k = k' \cdot p^{\alpha}$, where $\gcd(k', p) = 1$. It follows directly from Euler's theorem that $\gamma \mapsto p^{\gamma} \mod k'$ is periodic with period $\phi(k')$ and, clearly, $\gamma \mapsto p^{\gamma} \mod p^{\alpha}$ is eventually constant. Thus, $\gamma \mapsto p^{\gamma} \mod k$ is eventually periodic, by the Chinese Remainder Theorem.

    	\end{proof}

\subsection{Composite case}

    It remains to consider the case where $\lambda$ is composite. The main content of the argument is contained in the following analogue of Lemma \ref{lem:a(q^gamma)-ev-constant}.
	\begin{lemma}\label{lem:a(p^gamma)-ev-constant}
		Suppose that $\lambda$ is not a prime power.
		For any prime $p \mid \lambda$, the sequence
	\begin{align*}
		n \mapsto \chi_{p^{\alpha(p)}}\rb{\frac{n}{p^{\nu_{p}(n)}}}
    			\cdot  \frac{a(p^{\nu_{p}(n)})}{\chi(\overline{p})^{\nu_{p}(n)}}
	\end{align*}
	is periodic.
\end{lemma}

    \begin{proof}
	The proof is structured very similarly to the proof of Lemma~\ref{lem:a(q^gamma)-ev-constant}.
	The proof splits again into two parts.
	We first show that $\gamma \mapsto a(p^{\gamma})/\chi(\overline{p})^{\gamma}$ is eventually constant.
	Then we distinguish the case where this sequence eventually equals zero.
	If this is the case, then we are done by the second part of Lemma~\ref{lem_const_to_periodic}.
	If the sequence is not eventually equal to zero, then we also need to show that $\chi_{p^{\alpha(p)}}$ is trivial, i.e. constant and equal to $1$ on integers coprime to $p$.
	In this case we apply the first part of Lemma~\ref{lem_const_to_periodic} and see that
	\begin{align*}
		n \mapsto \frac{a(p^{\nu_{p}(n)})}{\chi(\overline{p})^{\nu_{p}(n)}}
	\end{align*}
	is periodic.
	
	To prove the claims, we consider the finite $\lambda$-kernel of $a$. It follows by the pigeonhole principle that there exist $k\geq 1$ and $r_1 < r_2 < (\lambda/p^{\alpha(p)})^k$ such that
	      \begin{align}\label{eq:862:14}
	        		r_1 \equiv r_2 \equiv 1 \bmod h \lambda
	      \end{align}
	      and
	      \begin{align}\label{eq:862:13}
	        a(n \lambda^{k} + p^{\alpha(p)k} r_1) = a(n \lambda^{k} + p^{\alpha(p)k} r_2)
	      \end{align}
	      for all $n \in \N$. (This is the only step, where we need the assumption that $\lambda$ is not a prime power, ensuring $\lambda/p^{\alpha(p)} > 1$.) 				We note that, by \eqref{eq:862:14}, we have $r_2-r_1=p^\beta r$ for some $\beta \geq \alpha(p)$ and $r$ with $p \nmid r$. 	
	
	Take any $\gamma \geq \alpha(p) + \beta$ and $0<s<p^{\alpha(p)}$ with $(s,p) = 1$. Since $(\lambda/p^{\alpha(p)}, p) = 1$ and $(\lambda, h) = 1$, we can find --- using the Chinese Remainder Theorem --- an integer $n$ such that
	\begin{align*}\label{eq_n_crt}
		\begin{split}
		n \cdot (\lambda/p^{\alpha(p)})^k + r_1 &\equiv s \cdot p^{\gamma} \bmod{p^{\gamma + \alpha(p)}},\\
		n \cdot \lambda^k + r_1 p^{\alpha(p) k} &\equiv 1 \bmod{h},
		\end{split}
	\end{align*}
	which can equivalently be written as
	\begin{equation}\label{eq:862:15}
		\begin{split}
		n \cdot \lambda^k + r_1 p^{\alpha(p)k} &\equiv s \cdot p^{\gamma+\alpha(p)k} \bmod{p^{\gamma + (k+1)\alpha(p)}},\\
		n \cdot \lambda^k + r_1 p^{\alpha(p) k} &\equiv 1 \bmod{h}.
		\end{split}
		\end{equation}

	Our goal is now again to compute $a(n \lambda^k + 	p^{\alpha(p)k} r_i)$ similarly to \eqref{eq_decompose_a_2}. Therefore, we will first determine $\nu_p(n \lambda^k + p^{\alpha(p)k} r_i) =: m_i$ and then show that $\rb{\rb{n \lambda^k + p^{\alpha(p) k} r_i}/{p^{m_i}}, h \lambda} = 1$. Thus, we will find by the multiplicativity of $a$,
	\begin{align*}
		a(n \lambda^k + p^{\alpha(p) k} r_i) = a(p^{m_i}) \cdot a\rb{\frac{n \lambda^k + p^{\alpha(p) k} r_i}{p^{m_i}}} =  a(p^{m_i}) \cdot \chi\rb{\frac{n \lambda^k + p^{\alpha(p) k} r_i}{p^{m_i}}}.
	\end{align*}
	Then we will decompose $\chi = \chi_{p^{\alpha(p)}} \cdot \chi_{\lambda/p^{\alpha(p)}} \cdot \chi_{h}$, so that it will only remain to determine the residues of $\rb{n \lambda^k + p^{\alpha(p) k} r_i}/{p^{m_i}}$ modulo $p^{\alpha(p)}, \lambda/p^{\alpha(p)}$ and $h$.

	We claim that $m_1 = \nu_p(n\lambda^k + p^{\alpha(p)k}r_1) = \gamma + \alpha(p)k$. Indeed, we have by \eqref{eq:862:15} (for \eqref{eq_n_mod_1} and \eqref{eq_n_mod_2}) and \eqref{eq:862:14} (for \eqref{eq_n_mod_3})
	      \begin{align}
	      	\frac{n \lambda^k + p^{\alpha(p) k} r_1}{p^{\gamma + \alpha(p) k}} &\equiv s \bmod p^{\alpha(p)}, \label{eq_n_mod_1}\\
		\frac{n \lambda^k + p^{\alpha(p) k} r_1}{p^{\gamma + \alpha(p) k}} &\equiv p^{-\gamma - \alpha(p) k} \bmod h, \label{eq_n_mod_2}\\
		\frac{n \lambda^k + p^{\alpha(p) k} r_1}{p^{\gamma + \alpha(p) k}} &\equiv p^{-\gamma} \bmod \lambda/p^{\alpha(p)}. \label{eq_n_mod_3}
	      \end{align}
	Thus, by \eqref{eq_n_mod_1}, \eqref{eq_n_mod_2}, \eqref{eq_n_mod_3}, the complete multiplicativity of Dirichlet characters and ~\eqref{kreska}, we obtain
	\begin{align*}
		a(n \lambda^k + p^{\alpha(p) k} r_1) &= a(p^{\gamma + \alpha(p) k}) \cdot \chi_{p^{\alpha(p)}}(s) \cdot \chi_{\lambda/p^{\alpha(p)}}(p^{-\gamma}) \cdot \chi_{h}(p^{-\gamma - \alpha(p) k})\\
			&= \chi_{p^{\alpha(p)}}(s) \cdot \frac{a(p^{\gamma + \alpha(p) k})}{\chi(\overline{p})^{\gamma + \alpha(p) k}} \cdot \chi_{\lambda/p^{\alpha(p)}}(p^{\alpha(p)k}).
	\end{align*}

We next consider $a(n\lambda^k + p^{\alpha(p)k}r_2)$. We claim that $m_2 = \nu_p(n \lambda^k + p^{\alpha(p)k}r_2) = \beta + \alpha(p) k$. Indeed, we have by $r_2 - r_1 = r p^{\beta}$,
\begin{align*}
	\frac{n \lambda^k + p^{\alpha(p)k} r_2}{p^{\beta + \alpha(p)k}} = \frac{n \lambda^k + r_1p^{\alpha(p) k} + r p^{\beta + \alpha(p) k}}{p^{\beta + \alpha(p)k}} \equiv r \bmod p^{\alpha(p)},
\end{align*}
as $\gamma \geq \beta + \alpha(p)$. Moreover, we find by $r_1 \equiv r_2 \bmod h \lambda$, and \eqref{eq_n_mod_2}, \eqref{eq_n_mod_3},
\begin{align}\label{eq:862:2}
\frac{n\lambda^k + p^{\alpha(p)k}r_2}{p^{\beta + \alpha(p) k}} &\equiv p^{-\beta - \alpha(p)k} \bmod h,\\
\frac{n\lambda^k + p^{\alpha(p)k}r_2}{p^{\beta + \alpha(p) k}} &\equiv p^{-\beta} \bmod \lambda/p^{\alpha(p)}.
\end{align}
Now, we have by the same arguments as above,
\begin{align*}
	a(n\lambda^k + p^{\alpha(p)k} r_2) &= a(p^{\beta + \alpha(p)k}) \cdot \chi_{p^{\alpha(p)}}(r) \cdot \chi_{h}(p^{-\beta - \alpha(p) k}) \cdot \chi_{\lambda/p^{\alpha(p)}}(p^{-\beta})\\
		&= \frac{a(p^{\beta + \alpha(p)k}) \cdot \chi_{p^{\alpha(p)}}(r) \cdot \chi_{\lambda/p^{\alpha(p)}}(p^{\alpha(p)k})}{\chi(\overline{p})^{\beta + \alpha(p)k}}.
\end{align*}
	Thus, we have in total by \eqref{eq:862:13},
	\begin{align*}
		\chi_{p^{\alpha(p)}}(s) \cdot \frac{a(p^{\gamma + \alpha(p) k})}{\chi(\overline{p})^{\gamma + \alpha(p) k}} = \frac{a(p^{\beta + \alpha(p)k}) \cdot \chi_{p^{\alpha(p)}}(r)}{\chi(\overline{p})^{\beta + \alpha(p)k}},
	\end{align*}
	for all $\gamma \geq \alpha(p) + \beta$ and $0<s<p^{\alpha(p)}$ with $(s,p) = 1$, where the right hand side of the above equality is independent of $\gamma$ and $s$. Thus the claims follow by the same reasoning as in the proof of Lemma~\ref{lem:a(q^gamma)-ev-constant}.
	\end{proof}

	As the product of periodic sequences is again periodic, we conclude the composite case.
	\begin{corollary}
		Suppose that $\lambda$ is not a prime power. Then $a$ is periodic and fulfills trivially Theorem~\ref{th_main}.
	\end{corollary}

		\section{Proof of Theorem~\ref{th_main} in the sparse case}
		Throughout this section we assume that $a(n) = 0$ for all $n > 1$ with $(n, \lambda h) = 1$.

		The proof of the following proposition can be seen as a variant of Schuetzenberger's proof \cite{Schuetzenberger1968} that no infinitely supported automatic sequence can only be supported on prime numbers.
		\begin{proposition}\label{pr_sparse_p}
			Suppose that $\alpha \mapsto a(p^{\alpha})$ is not finitely supported for some prime $p$.
			Then $a$ is $p$-automatic.
		\end{proposition}
		\begin{proof}
			Assume that $\alpha\mapsto a(p^\alpha)$ is not finitely supported. Then $p$ cannot be coprime with $h\lambda$ as we are in the sparse case. So $p|h\lambda$.
There are two cases then: either $p|h$ and we will show that this is in fact impossible, or $p|\lambda$ in which case we will show that $a$ is $p$-automatic.
			
			Suppose first that there exists $q \mid h$ such that $a(q^{\alpha}) \neq 0$ for infinitely many $\alpha \in \N$.
			Thus, we can apply Lemma~\ref{le_pumping_automatic} (the pumping lemma) to $q^i$ for some large $i$.
			Therefore, there exists a decomposition $q^i = x \lambda^{k+\ell} + y \lambda^{k} + z$ (with $\ell\geq1$) such that
			\begin{align}\label{eq_pump_q}
				a(q^i) = a\left(x \lambda^{k+(n+1)\ell}+ y \lambda^{k} \rb{\lambda^{0} + \lambda^{\ell} + \ldots + \lambda^{n \ell}} + z \right) \neq 0
			\end{align}
			for all $n \in \N$.
			Since $\lambda^\ell$ is coprime with $h$ (hence with $q^i$), it is also coprime with $(\lambda^\ell-1)hq^i$. Therefore, for some $L\geq 1$, we have $\lambda^{L\ell}\equiv 1$ mod~$(\lambda^\ell-1)hq^i$. It follows that
			\[ \lambda^{0} + \lambda^{\ell} + \ldots + \lambda^{(L-1)\ell} = (\lambda^{L\ell} - 1)/(\lambda^\ell - 1) \equiv 0 \bmod{hq^i}.
			 \]
			Let us consider the integer
			\[
				m := x  \lambda^{k + (L+1) \ell}+ y \lambda^{k} \rb{\lambda^{0} + \lambda^{\ell} + \ldots + \lambda^{L \ell}} + z.
			\]			
			From the definition of $m$ it follows that for some $r\geq1$, we have
			\begin{align*}
				m&=x\lambda^{k+\ell}\lambda^{L\ell}+y\lambda^k(rhq^i+\lambda^{L\ell})+z\\
					&=q^i + x\lambda^{k+\ell}(\lambda^{L\ell}-1) + y\lambda^k(rhq^i+\lambda^{L\ell}-1).
			\end{align*}
			Therefore (as clearly $\lambda^{L\ell}\equiv 1$ mod~$h\lambda$),
			\begin{equation}\label{eq:678:1}
				m \equiv q^i \bmod{hq^i}.
			\end{equation}
			Moreover, as $z\equiv q^i$ mod~$\lambda$, we also have
			\begin{equation}\label{eq:678:2}
				m \equiv q^i \bmod \lambda.
			\end{equation}
	From the two above congruences \eqref{eq:678:1}, \eqref{eq:678:2} it follows that $m \equiv q^i \bmod q^i h \lambda$ and consequently $m/q^i \equiv 1 \bmod h \lambda$.
			Thus, $(q^i, m/q^i) = 1$ and, by the multiplicativity of $a$, we have $a(m) = a(q^i) \cdot a(m/q^i)$.
			Since $(m/q^i, h \lambda) = 1$ and $m/q^i > 1$, we have $a(m/q^i) = 0$, as we are in the sparse case. But by~\eqref{eq_pump_q}, $a(q^i)=a(m)\neq0$ which leads to a contradiction.
			
			Suppose now that $a(p^{\alpha}) \neq 0$ for infinitely many $\alpha$ for some prime $p \mid \lambda$.
			By the same reasoning as before, we find that for some large enough $i$, we can find a decomposition $p^i = x \lambda^{k + \ell}+ y \lambda^{k} + z$ (with $\ell\geq1$) such that
			\begin{align}\label{eq_pump_p}
				a(p^i) = a(x \lambda^{k + (n+1) \ell} + y \lambda^{k} \rb{\lambda^{0} + \lambda^{\ell} + \ldots + \lambda^{n \ell}} + z) \neq 0
			\end{align}
			for all $n \in \N$. 
			Similarly to the previous  case, note that $\lambda^\ell$ is coprime with $(\lambda^\ell-1)h$, so for some $L\geq1$, we have $								\lambda^{L\ell}=1$ mod~$(\lambda^\ell-1)h$, and so also  $\lambda^{nL\ell}=1$~mod $(\lambda^{\ell}-1)h$ for each $n\geq1$. As a consequence,
			\[
				\lambda^0+ \lambda^{\ell} + \dots + \lambda^{nL\ell} \equiv 1 \bmod{h}
			\]
		for all integers $n \in \N$. Thus,
			\begin{align}\label{eq_mn_1}
			\begin{split}
				m(n) &:= x \lambda^{k + (nL+1) \ell} + y \lambda^{k} \rb{\lambda^{0} + \lambda^{\ell} + \ldots + \lambda^{nL \ell}} + z\\
					& \equiv x \lambda^{k + \ell} + y \lambda^k + z = p^i \bmod h.
			\end{split}
			\end{align}
			In particular, $m(n)$ is coprime to $h$ for all $n \in \N$.
			Moreover,
			\begin{align}\label{eq_mn_2}
				m(n) \equiv z \equiv p^i \bmod \lambda,
			\end{align}
			so $m(n)$ is coprime to $h\lambda/p^{\alpha(p)}$.
			
			Since we have $a(m(n)) = a(p^i) \neq 0$ by~\eqref{eq_pump_p} and we are in the sparse case, all prime factors of $m(n)$ divide $h \lambda$. But  by~\eqref{eq_mn_1} and \eqref{eq_mn_2}, we obtain that  $p$ is the only common prime factor of $m(n)$ and $h\lambda$. It follows that for each $n \in \N$ there exists an integer $k(n)$ such that $m(n) = p^{k(n)}$. We can estimate $k(n)$ by
			\begin{align*}			k(n) &= \log_p( m(n) ) = \log_p \left(x \lambda^{k + (nL+1) \ell} + y \lambda^k \frac{\lambda^{(nL+1)\ell} - 1}{\lambda^\ell - 1} + z \right)
			\\& = (k + (nL+1) \ell)\log_p(\lambda) + \log_p \left(x +  \frac{y}{\lambda^\ell - 1}  \right) + o(1)
			\end{align*}
as $n \to \infty$. In particular, the sequence $(k + (nL+1) \ell)\log_p(\lambda) \bmod 1$ converges (in $\R/\Z$) as $n \to \infty$. This is only possible if $\log_p(\lambda) \in \Q$, since otherwise the above sequence would be equidistributed. Consequently, $\lambda$ is a power of $p$ and $a$ is $p$-automatic.
		\end{proof}
		
		\begin{proof}[Proof of Theorem~\ref{th_main} in the sparse case]
			Let us consider the set
			\[ P := \{p \in \P: \alpha \mapsto a(p^{\alpha}) \text{ is not finitely supported} \}.\]
			We distinguish the following cases:\\
			\begin{itemize}[wide]
				\item $\abs{P} \geq 2$: Let $p_1, p_2 \in P$ be distinct. Using Proposition~\ref{pr_sparse_p}, we find that $a$ is $p_1$-automatic and $p_2$-automatic. By Cobham's Theorem~\ref{th_cobham}, $a$ is eventually periodic.
				Moreover, by Lemma~\ref{le_eventually_periodic}, $a$ is either periodic or finitely supported.
				In both cases, $a$ is $p$-automatic for every prime $p$ and we can write it in the form $a(n)=f_1(\nu_p(n))f_2(n/p^{\nu_p(n)})$ (cf.{} Proposition \ref{lem:unique}) with $f_2(n)= a(n)$ for $n$ not divisible by $p$ and $f_1(k) = a(p^k)$ which is the form required by Theorem~\ref{th_main} in view of Lemma~\ref{podciag}.
				
				\item $P = \emptyset$: Since we are in the sparse case, all primes $p$ such that $a(p^\alpha) \neq 0$ for some $\alpha$ are divisors of $h\lambda$. For each prime divisor $p$ of $h\lambda$, there are finitely many exponents $\alpha$ such that $a(p^\alpha) \neq 0$. Hence, there are only finitely many prime powers on which $a$ takes non-zero values. Since $a$ is multiplicative, it follows that $a$ has finite support. It is now easy to write it in the form of Theorem~\ref{th_main} (with $f_1$ and $f_2$ eventually zero).
				
				\item $P = \{p\}$: By Proposition~\ref{pr_sparse_p}, $a$ is $p$-automatic. We can write $a$ as $a(n)=f_1(\nu_{p}(n))f_2(n/p^{\nu_{p}(n)})$, where $f_1(k) = a(p^k)$ is eventually periodic and $f_2$ vanishes on multiples of $p$ and agrees with $a$ elsewhere. In particular, $f_2$ is a sparse automatic multiplicative sequence such that $\alpha \mapsto f_2(q^\alpha)$ has finite support for all primes $q$. By the same argument as in the case $P = \emptyset$, $f_2$ is finitely supported, and hence eventually periodic.
			\end{itemize}
%
%
%
		\end{proof}
		
\section{Remarks}\label{s:Remarks}

\subsection{Completely multiplicative case}
We now derive the main result of \cite{Li2019} from Theorem~\ref{th_main}:

\begin{proposition} Assume that $a$ is a completely multiplicative automatic sequence. Then there exists a prime $p$ such that $$a(n)=\epsilon^{\nu_p(n)}\chi(n/p^{\nu_p(n)}),\;n\in\N,$$ for a root of unity $\epsilon$ and $\chi$  a Dirichlet character  or the support of $a$ is contained in $\{p^k:\:k\geq 0\}$.
\end{proposition}
\begin{proof}
	By Theorem~\ref{th_main} (and Lemma~\ref{le_eventually_periodic}), we can write $a(n) = f_1(\nu_p(n)) \cdot f_2\rb{{n}/{p^{\nu_p(n)}}}$, where $f_2$ is either finitely supported and multiplicative or periodic and multiplicative.
	As $a$ is completely multiplicative, we note that $f_1(k) = f_1(1)^{k}$ and $f_1(1)$ is either a root of unity $\epsilon$ or $0$ as $a$ takes only finitely many values.
	In particular, we have that $a$ is periodic if $f_1(1) = 0$.
	As $a$ is  completely multiplicative, $f_2$ is completely multiplicative because $f_2$ vanishes on the multiples of $p$.
	As $f_2$ is eventually periodic, it follows by~\cite[Proposition 2.2]{Allouche2018} that $f_2$ is either a Dirichlet character or $f_2(n) = 0$ for all $n > 1$. This finishes the proof.
\end{proof}

\subsection{Associated dynamical systems}		
	
When $a:\N \to\C$ is automatic and $A=a(\N)$ is the ``alphabet'' of $a$, then we let $X_a\subset A^{\Z}$	 denote the subshift generated by $a$.\footnote{The set $X_a$ consists of all sequences $x$ in $A^{\Z}$ such that each subword of $x$ appears in $a$.} We hence obtain a topological dynamical system $(X_a,S)$, where $S$ stands for the left shift, $S\big( (x_n)_{n \in \Z} ) = (x_{n+1})_{n \in \Z}$. We recall that $a$ is primitive if and only if $(X_a,S)$ is minimal. In fact, this condition is further equivalent to the assertion that $(X_a,S)$ is strictly ergodic, in which case the unique invariant measure is denoted by $\mu_a$.
	Such subshifts originating from primitive automatic sequences have been thoroughly studied (see for example~\cite{Queffelec2010}).

\begin{remark} In general, the dynamical system generated by an automatic sequence need not be uniquely ergodic, see again for example \cite{Queffelec2010}. However, automatic multiplicative sequences generate uniquely ergodic systems. Indeed,  in the sparse case we have only finitely many primes $p_1,\ldots,p_k$ such that
the support of $a$ is contained in the set
$$
D:=\{p_1^{\alpha_1}... p_k^{\alpha_k}:\: \alpha_i\in\N_0,\;i=1,\ldots,k\}$$ which already follows from Corollary~\ref{c:KonMul}. Now, the set $D$ has upper Banach density zero, so the only invariant measure for the system $(X_a,S)$ is given by the fixed point given by the all zero sequence. In the dense case, by the comments after Lemma~\ref{l:prz3}, $a$ is Toeplitz. Thus, we have minimality for $(X_a, S)$, which implies unique ergodicity for automatic sequences, see for example \cite{Queffelec2010}.\end{remark}

\begin{remark}\label{rmk-M(a)-exists}
 Since $(X_a,S)$ is always uniquely ergodic, all points in $X_a$ are generic. It follows that $a$ itself has to have a mean. In other words the limit
$$
M(a)=\lim_{N\to\infty}\frac1N\sum_{n=1}^Na(n)$$
does exist. We will compute its value in the next subsection.\end{remark}

\subsection{Averages}

Let us briefly discuss the average $M(a)$ of an automatic multiplicative sequence $a(n)$. Recall that $M(a)$ is guaranteed to exists (cf.{} Remark \ref{rmk-M(a)-exists}).

\begin{proposition}
	Let $(a(n))_{n \geq 0}$ be an automatic multiplicative sequence.
	Then $M(a)$
	is given by
	\begin{align}\label{eq_average}
		M(a) = M(f_2) \cdot \sum_{k\geq 0} \frac{f_1(k)}{p^{k}},
	\end{align}
	where $f_1$, $f_2$ and $p$ are given by Theorem~\ref{th_main}.
\end{proposition}
\begin{proof}
For $\alpha \geq 0$ and $0 \leq r < p$, consider the infinite arithmetic progression $P(\alpha,r) = p^{\alpha+1} \N_0 + p^\alpha r$. Since $\bigcup_{\alpha=0}^{\infty} \bigcup_{r=1}^{p-1} P(\alpha,r) = \N$ and $\sum_{\alpha=0}^{\infty} \sum_{r=1}^{p-1} d(P(\alpha,r)) = 1$, it is an elementary exercise in analysis that
\begin{equation}\label{eq_average-1}
	M(a) = \sum_{\alpha=0}^{\infty} \sum_{r=1}^{p-1} M\left(a 1_{P(\alpha,r)}\right),
\end{equation}
where $1_X$ denotes the characteristic sequence of the set $X$ and $d(X) = M(1_X)$ denotes its density. For each $\alpha \geq 0$ and $1 \leq r < p$ we can compute
\begin{align*}
	M\left(a 1_{P(\alpha,r)}\right)
	&= \lim_{N \to \infty} \frac{1}{p^{\alpha} N} \sum_{n=0}^{N-1} a\left(p^{\alpha+1} n+p^{\alpha}r\right)
	\\& = \frac{f_1(\alpha)}{p^\alpha} \lim_{N \to \infty} \frac{1}{N} \sum_{n=0}^{N-1} f_2\left( pn+r\right).
\end{align*}	
Note that the last limit exists because $f_2$ is eventually periodic. Since $f_2(pn) = 0$ for all $n \geq 0$, we also have $M\left(a 1_{P(\alpha,0)}\right) = 0$. Combining the above formulae, for any $\alpha \geq 0$, we conclude that
\begin{equation}\label{eq_average-2}
	\sum_{r=1}^{p-1} M\left(a 1_{P(\alpha,r)}\right) = \frac{f_1(\alpha)}{p^\alpha} M(f_2).
\end{equation}
Now, \eqref{eq_average} follows by substituting \eqref{eq_average-2} in \eqref{eq_average-1}.
\end{proof}

\begin{remark} Note $M(a) = 0$ if $M(f_2) = 0$, which in particular includes the case when $f_2$ is finitely supported (see Lemma \ref{le_eventually_periodic}). It is also possible that $M(a)=0$ even though $M(f_2) \neq 0$. Indeed, this is the case for instance when $p=2$, $f_1$ is given by $f_1(0)=1$, $f_1(k)=-1$ for all $k\geq1$, and $f_2$ is any principal character.

Finally, note that in the dense case $a$ is never aperiodic because $a$ is Toeplitz (see the remarks after Lemma~\ref{l:prz3}).

\end{remark}

		\bibliographystyle{abbrv}
 		\bibliography{bibliography_general}
\end{document}